\documentclass[11pt]{amsart}

\usepackage{amssymb,latexsym} 
\usepackage{enumitem}

\usepackage{graphicx}

\makeatletter
\@namedef{subjclassname@2020}{%
  \textup{2020} Mathematics Subject Classification}
\makeatother
\usepackage[T1]{fontenc}
\newtheorem{theorem}{Theorem}[section]

\newtheorem{lemma}[theorem]{Lemma}
\newtheorem{proposition}[theorem]{Proposition}

\theoremstyle{definition}

\newtheorem{remark}[theorem]{Remark}

\numberwithin{equation}{section}
\begin{document}


\baselineskip=17pt


\title[Hecke eigenvalues over primes in short intervals]{On sum of Hecke eigenvalue squares over primes in very short intervals}

\author{Jiseong Kim}
\address{University at Buffalo, Department of Mathematics
244 Mathematics Building
Buffalo, NY 14260-2900}
\email{Jiseongk@buffalo.edu}

\date{}

\begin{abstract} 
Let $\eta>0$ be a fixed positive number, let $N$ be a sufficiently large number.
In this paper, we study the second moment of the sum of Hecke eigenvalues over primes in short intervals (whose length is $\eta \log N$) on average (with some weights) over the family of weight $k$ holomorphic Hecke cusp forms. We also generalize the above result to Hecke-Maass cusp forms for $SL(2,\mathbb{Z})$ and $SL(3,\mathbb{Z}).$ By applying the Hardy-Littlewood prime 2-tuples conjecture, we calculate the exact values of the mean values.
\end{abstract}



\maketitle

\section{Introduction}\

Let $\mathbb{H}=\{z=x+iy | x \in \mathbb{R}, y \in (0,\infty)\}, G=SL(2,\mathbb{Z}).$ 
Define $j_{\gamma}(z)=(cz+d)^{-1}$ where 
$\gamma= \left(\begin{matrix}
a & b \\
c & d 
\end{matrix} \right) \in G.$
When a holomorphic function $f: \mathbb{H} \rightarrow \mathbb{C}$ satisfies
\begin{equation}f(\gamma z)= j_{\gamma}(z)^{k}f(z)\nonumber\end{equation}
for all $\gamma \in SL(2,\mathbb{Z}),$ it is called a modular form of weight $k.$ It is well known that any modular form $f(z)$ has a Fourier expansion at the cusp $\infty$
\begin{equation}
    f(z)=\sum_{n=0}^{\infty} b_{n}e(nz)
\end{equation}
where $e(z)=e^{2\pi iz},$ and the normalized Fourier coefficient $a(n)$ of $f(z)$ is defined by
\begin{equation}
a(n):=b_{n}n^{-\frac{k-1}{2}}.
\end{equation}

 The set of all modular forms of fixed weight $k$ is a vector space and we use  $M_{k}$ to denote this. And we denote the set of all modular forms in $M_{k}$ which have a zero constant term by $C_{k}.$ 
Let $T_{n}$ be the n-th Hecke operator on $C_{k},$ which means 
$$ (T_{n}f)(z):= \frac{1}{\sqrt{n}} \sum_{ad=n} \sum_{b(\rm{mod} \thinspace \it{d})} f(\frac{az+b}{d})$$
for all $f \in C_{k}.$

It is known that there is an orthonormal basis of $C_{k}$ which consists of eigenfunctions for all Hecke operators $T_{n},$ and these are called Hecke cusp forms. In this paper, we use $S_{k}$ to refer this orthonormal basis of $C_{k}.$
When $f$ is a Hecke cusp form, the eigenvalues $\lambda_{f}(n)$ of the $n$-th Hecke operator satisfy 
$$a(n)=a(1)\lambda_{f}(n).$$ 
For details, see \cite[Chapter 14]{IK1}. 

 For convenience, we assume that $k$ is always even natural number and summing over the index $p$ denotes summing over primes. $P$ denotes the set of all prime numbers.
 
For Hecke-Maass cusp forms $\psi$ for $SL(2,\mathbb{Z}),$ Y. Motohashi \cite{Motohashi} proved that there exist constants $c_{0},\theta_{0}>0$ such that uniformly for $(\log N)^{-\frac{1}{2}} \leq \theta \leq \theta_{0},$ 
\begin{equation}
    \sum_{N-y \leq p \leq N} \lambda_{\psi}(p)^{2} = \frac{y}{\log N}(1+O_{\psi}(e^{-\frac{c_{0}}{\theta}})), \thinspace  y=N^{1-\theta}
\end{equation}
for sufficiently big $N.$ Note that the main term $\frac{y}{\log N}$ in (1.3) is also the main terms of the number of primes in $[N-y,N].$ For the number of prime numbers in short intervals $[N-y,N],$ there are many results for various $y$ (see \cite{Yildirim}).  For the very short range such as $y=\eta \log N$ for any fixed $\eta>0,$ P. X. Gallagher \cite{Gallagher} proved that by assuming the Hardy-Littlewood prime $j$-tuples conjecture and using the method of moments, the number of primes in $[N-y,N]$ has a Poisson distribution $P(\eta).$  
 
Let $\pi_{\vec{d}}(N)$ be the number of positive integers $n \leq N$ such that $n-d_{1},n-d_{2},...,n-d_{j}$ are all prime where $\vec{d}=(d_{1},d_{2},...,d_{j}) \in \mathbb{N}^{j}$ and $1\leq  d_{1}<d_{2}<...<d_{j}\leq y.$ And let $\pi(N)$ be the number of primes $p \leq N.$ 
The Hardy-Littlewood prime $j$-tuples conjecture claims that 

\begin{equation}\pi_{d}(N) \sim \prod_{p} \frac{p^{j-1}}{(p-1)^{j}}(p-v_{d}(p)) \frac{N}{\log^{j}(N)}\end{equation}
where $v_{d}(p)$ is the number of distinct residue classes in $\{d_{1},d_{2},...,d_{j}\}$ modulo $p.$
 When $y=\eta \log N$ for some fixed positive number $\eta,$ P. X. Gallagher showed that by assuming (1.4) for all $j\in \mathbb{N},$

\begin{equation}\begin{split}\sum_{n\leq N} (\pi(n)-\pi(n-y))^{j}&= \sum_{n \leq N} \sum_{n-y < p_{1},...,p_{j} \leq n} 1 
\\&=\sum_{r=1}^{j} \sigma(j,r) \sum_{\vec{d}} \pi_{\vec{d}}(N)
\\&= N(m_{j}(\eta)+o(1))
\end{split}\end{equation}
where $\sigma(j,r)$ is the number of maps form $\{1,2,...,j\}$ onto $\{1,...,r\}$,  the inner sum over $\vec{d}$ means, $\vec{d}=(d_{1},d_{2},...,d_{r}) \in \mathbb{N}^{r},$ $1\leq  d_{1}<d_{2}<...<d_{r}\leq y.$ And
$$m_{j}(\eta)= \sum_{r=1}^{j} \eta^{r}S(r,j)$$ 
where $S(r,j)$ is the Stirling number of the second kind (the number of ways to partition a set $\{1,2,3,..,r \}$ into $j$ nonempty unlabelled subsets). For each $1\leq r \leq j,$ $\eta^{r} S(r,j)$ corresponds to   $\sigma(j,r)$ times the sum of $\pi_{\vec{d}}(N)$ over $\vec{d}$ where $\vec{d}=(d_{1},d_{2},...,d_{r}) \in \mathbb{N}^{r},$ $1\leq  d_{1}<d_{2}<...<d_{r}\leq y.$  Because of (1.3), one might wonder whether $\sum_{p \leq N} \lambda_{f}(p)^{2}$ acts similar to $\pi(N)$ in the sense of (1.5).
In this paper, we show that 
\begin{equation}\begin{split}& \sum_{f \in S_{k}}\frac{1}{\| f \|^{2}}\sum_{n \leq N}\Big(\sum_{n-\eta \log N\leq p \leq n} \lambda_{f}(p)^{2}\Big)^{2}
\\&= 
\sum_{f \in S_{k}} \frac{1}{\| f \|^{2}} \Big(\sum_{n \leq N} \big((\pi(n)-\pi(n-y))^{2}+ \eta+o_{\epsilon}(1)\big)\Big)\end{split}\end{equation}
for sufficiently big $k$ (depends on $N$) and sufficiently big $N,$ where  $\|f\|$ is the Petersson norm of $f$ over $C_{k}.$ (1.6) suggests that for intervals which have constant multiple of the average prime-gap length, $\lambda_{f}(p)^{2}$ act slightly different from primes. 
And by  (1.5) (assuming the Hardy-Littlewood $2$-tuples conjecture), the right-hand side of (1.6) is
\begin{equation}\sum_{f \in S_{k}} \frac{N}{\| f \|^{2}} \Big(m_{2}(\eta)+\eta+o_{\epsilon}(1)\Big),\nonumber\end{equation} 
and by applying the Petersson Trace formula (see Lemma 2.2, (2.5), this is $$N\frac{(4\pi)^{k-1}}{\Gamma(k-1)}\big(m_{2}(\eta)+\eta+o_{\epsilon}(1)\big).$$
\begin{remark}
For $j=2,$ $m_{2}(\eta)=\eta^{2}+\eta.$ For the average of higher powers $(j>2)$ as (1.6), one might expect that 
\begin{equation}\begin{split}\sum_{n \leq N}&\Big(\sum_{n-\eta \log N \leq p \leq n} \lambda_{f}(p)^{2}\Big)^{j} \\& \sim \sum_{n \leq N}\Big(\big| \pi(n- \eta \log N) - \pi(n)\big|^{j}+\big(\frac{(2j)!}{(j!)(j+1)}-S(j,j)\big)\eta^{j}\Big)\nonumber\end{split}\end{equation} for some different range of $k$ (in the sense of (1.7)). For this, we need the $2j$-th power versions of Lemma 2.3
\begin{equation}  \sum_{f \in S_{k}} \frac{1}{\|f\|^{2}} \sum_{p \leq N} \big(\lambda_{f}(p)\big)^{2j} =  \sum_{f \in S_{k}} \frac{1}{\|f\|^{2}} \big(\frac{(2j)!}{(j!)(j+1)}+o(1)\big) \frac{N}{\log N}, \end{equation}
and this is easily deduced from Hecke relations (see (2.1)).

\end{remark}
The idea of the proof of (1.6) is very simple. By applying Hecke relations, we separate 1 from some powers of $\lambda_{f}(n),$ then using the Petersson trace formula (Lemma 2.1) to treat non-constant terms. Note that the average difference $\eta$ in (1.6) comes from the digonal terms. Due to this simplicity, we get similar results like (1.6) for Hecke-Maass cusp forms for $SL(2,\mathbb{Z})$ and $SL(3,\mathbb{Z})$ (for the background of Hecke-Maass cusp form, see \cite[Chapter 3, Chapter 6]{DG}). For these, we only need the following reults come from direct applications of the Kuznetsov formula.

\begin{lemma} 
Let $\{\psi_{j}\}$ be an orthonormal basis of Hecke-Maass cusp forms for $SL(2,\mathbb{Z}),$ and let $\frac{1}{4}+t_{j}^{2}$ be the Laplace eigenvalue of $\psi_{j}.$
Let $p,q$ be distinct primes. Let $T>1,$ let $\epsilon>0$ be fixed small positive number.
Then 
\begin{equation}\begin{split}
&\sum_{j} \lambda_{\psi_{j}}(p^{2}) \frac{\zeta(2)}{L(1,\mathrm{sym}^{2}\psi_{j})} e^{-\frac{t_{j}}{T}}=O_{\epsilon}(T^{1+\epsilon}p^{\epsilon} +p^{\frac{1}{2}+\epsilon}),
\\& \sum_{j} \lambda_{\psi_{j}}(p^{2})^{2} \frac{\zeta(2)}{L(1,\mathrm{sym}^{2}\psi_{j})} e^{-\frac{t_{j}}{T}}=\frac{T^{2}}{6}+O_{\epsilon}(T^{1+\epsilon}p^{\epsilon}+p^{1+\epsilon}),
\\& \sum_{j} \frac{\zeta(2)}{L(1,\mathrm{sym}^{2}\psi_{j})} e^{-\frac{t_{j}}{T}}=\frac{T^{2}}{6}+O_{\epsilon}(T^{1+\epsilon}),
\\&  \sum_{j} \lambda_{\psi_{j}}(p^{2}) \lambda_{\psi_{j}}(q^{2}) \frac{\zeta(2)}{L(1,\mathrm{sym}^{2}\psi_{j})} e^{-\frac{t_{j}}{T}}=O_{\epsilon}(T^{1+\epsilon}(pq)^{\epsilon}+(pq)^{1+\epsilon})
\end{split}\end{equation}
where the summation over $j$ denotes the summation over all $\psi_{j}.$
\end{lemma}
\begin{proof}
See \cite[Lemma 1]{BB1}.
\end{proof}
By applying Lemma 1.2 with the methods in the proof of Proposition 3.1, we get the following theorem.
\begin{theorem}Assume the Hardy-Littlewood prime $2$-tuples conjecture. Let $\epsilon>0, \eta>0$ be fixed positive numbers. Let $N$ be a sufficiently big number, and $y=\eta \log N.$ Let $T$ be a positive number such that $\thinspace N^{1+\epsilon}=o_{\epsilon}(T).$
Then
\begin{equation} \begin{split}
\frac{1}{N} \sum_{j} \frac{\zeta(2)}{L(1,\mathrm{sym}^{2}\psi_{j})} & e^{-\frac{t_{j}}{T}} \sum_{n \leq N} \Big(\sum_{n-\eta \log N\leq p \leq n} \lambda_{\psi_{j}}(p)^{2}\Big)^{2}  \\&=  \sum_{j} \frac{\zeta(2)}{L(1,\mathrm{sym}^{2}\psi_{j})}e^{-\frac{t_{j}}{T}}\big(m_{2}(\eta)+\eta+o_{\epsilon}(1)\big) \\&=\frac{T^{2}}{6}\big(m_{2}(\eta)+\eta+o_{\epsilon}(1)\big).
\nonumber  
\end{split}\end{equation}
\end{theorem}

 Let $\{ \phi_{j}\}$ be an orthonormal basis of Hecke-Maass cusp forms for $SL(3,\mathbb{Z}),$ and let $\{A_{j}(n,1)\}$ be the Hecke eigenvalues.
Let $v_{j}$ be the Laplacian eigenvalue of $\phi_{j}.$  By Hecke relations,
\begin{equation} \begin{split} &|A_{j}(p,1)|^{2}=A_{j}(p,p)+1,\\& |A_{j}(p,1)|^{4}=A_{j}(p,p)^{2}+2A_{j}(p,p)+1 \end{split} \end{equation} (see  \cite[Theorem 6.4.11]{DG}). 
We need the following result comes from the GL(3)-Kuznetsov formula.
\begin{lemma} Let $p,q$ be distinct primes, let $T >1.$ Then
\begin{equation} \begin{split}
&\sum_{j} \frac{A_{j}(p,p)^{2}}{\underset{s=1}{\mathrm{Res}}\; L(s,\phi_{j} \times \bar{\phi}_{j})} e^{-\frac{v_{j}}{T^{2}}} = \frac{\sqrt{3}}{2^{7} \pi^{\frac{9}{2}}} T^{5}+ O_{\epsilon}(p^{4+\epsilon}T^{\frac{37}{8}+\epsilon}),
\\& \sum_{j} \frac{A_{j}(p,p)}{\underset{s=1}{\mathrm{Res}}\; L(s,\phi_{j} \times \bar{\phi}_{j})} e^{-\frac{v_{j}}{T^{2}}} = O_{\epsilon}(p^{2+\epsilon}T^{\frac{37}{8}+\epsilon}),
\\& \sum_{j} \frac{1}{\underset{s=1}{\mathrm{Res}}\; L(s,\phi_{j} \times \bar{\phi}_{j})} e^{-\frac{v_{j}}{T^{2}}} = \frac{\sqrt{3}}{2^{7} \pi^{\frac{9}{2}}} T^{5}+ O_{\epsilon}(T^{\frac{37}{8}+\epsilon})
\\& \sum_{j} \frac{A_{j}(p,p)A_{j}(q,q)}{\underset{s=1}{\mathrm{Res}}\; L(s,\phi_{j} \times \bar{\phi}_{j})} e^{-\frac{v_{j}}{T^{2}}} = O_{\epsilon}((pq)^{2+\epsilon}T^{\frac{37}{8}+\epsilon}).
\end{split} \end{equation}
where the summation over $j$ denotes the summation over all $\phi_{j},$ $\underset{s=1}{\mathrm{Res}}$ denotes the residue at $s=1.$

\end{lemma}
\begin{proof}
See \cite[Theorem 5]{BB1}.
\end{proof}

By applying Lemma 1.4 with the methods in the proof of Proposition 3.1, we get the following theorem.
\begin{theorem}Assume the Hardy-Littlewood prime $2$-tuples conjecture. Let $\epsilon>0, \eta>0$ be fixed positive numbers. Let $N$ be a sufficiently big number, and $y=\eta \log N.$ Let $T$ be a positive number such that $\thinspace N^{\frac{32}{3}+2\epsilon}=o_{\epsilon}(T).$
Then
\begin{equation} \begin{split}
\frac{1}{N} \sum_{j}\frac{1}{\underset{s=1}{\mathrm{Res}}\; L(s,\phi_{j} \times \bar{\phi}_{j})} e^{-\frac{v_{j}}{T^{2}}} &\sum_{n \leq N} \Big(\sum_{n-\eta \log N\leq p \leq n} |A_{j}(p,1)|^{2}\Big)^{2}  \\&=  \sum_{j} \frac{1}{\underset{s=1}{\mathrm{Res}}\; L(s,\phi_{j} \times \bar{\phi}_{j})} e^{-\frac{v_{j}}{T^{2}}}\big(m_{2}(\eta)+\eta+o_{\epsilon}(1)\big) \\&= \frac{\sqrt{3}}{2^{7} \pi^{\frac{9}{2}}} T^{5}\big(m_{2}(\eta)+\eta+o_{\epsilon}(1)\big).
\nonumber  
\end{split}\end{equation}
\end{theorem}
\section{Lemmas}
 By Hecke relations, we have
\begin{equation} \lambda_{f}(p)^{2}= \lambda_{f}(p^{2}) + 1.\end{equation}
Later, we need to deal with some off-diagonal terms $\lambda_{f}(p^{2})\lambda_{f}(q^{2})$ where $p\in P, \thinspace q \in P, \thinspace  q \neq p$ or $p\in P,\thinspace q=1.$  For this, we will use the following lemma.
\begin{lemma} {\bf (Trace formula)}
For any two natural numbers $m$ and $n,$
\begin{equation}\begin{split}
\sum_{f \in S_{k}} \frac{\lambda_{f}(n) \lambda_{f}(m)}{\|f\|^{2}}&=\frac{(4\pi)^{k-1}}{\Gamma(k-1)}\delta(m-n)\\&+O\Big(\frac{(4\pi)^{k-1}}{\Gamma(k-1)}\big( (\log (3mn))^{2} \frac{d\big((m,n)\big)(mn)^{\frac{1}{4}}}{k^{\frac{1}{2}}}\big) \Big)   
\end{split}\end{equation}
where $\delta$ is the delta function, $d$ is the divisor function, the implied constant is absolute, and $\|f\|$ is the Petersson norm of $f$ over $C_{k}.$
\end{lemma}
\begin{proof}
See \cite[Corollary 14.24, Theorem 16.7]{IK1}.
\end{proof}

By Lemma 2.1, we get the following Lemma.
\begin{lemma} Let $p$ be a prime. Let $f \in S_{k}.$ Then
 \begin{equation} \sum_{f \in S_{k}} \frac{\lambda_{f}(p^{2})}{ \|f\|^{2}}=O\Big(\frac{(4\pi)^{k-1}}{\Gamma(k-1)}\big( (\log (3p^{2}))^{2} \frac{p^{\frac{1}{2}}}{k^{\frac{1}{2}}}\big) \Big), \end{equation}
\begin{equation}\sum_{f \in S_{k}} \frac{\lambda_{f}(p^{2})^{2}}{\|f\|^{2}} =\frac{(4\pi)^{k-1}}{\Gamma(k-1)}+O\Big(\frac{(4\pi)^{k-1}}{\Gamma(k-1)}(\log (3p^{4})^{2}\frac{p}{k^{\frac{1}{2}}}\Big),\end{equation}
\begin{equation} \sum_{f \in S_{k}} \frac{1}{\|f\|^{2}} =\frac{(4\pi)^{k-1}}{\Gamma(k-1)}+O\Big(\frac{(4\pi)^{k-1}}{k^{\frac{1}{2}}\Gamma(k-1)}\Big). \end{equation}

\end{lemma}
\begin{proof}
For (2.3), put $n=p^{2}, m=1$ in (2.2). For (2.4), put $m=n=p^{2}$ in (2.2). For (2.5), put $m=n=1$ in (2.2). 
\end{proof}

\begin{lemma}  Let $N>0$ be sufficiently big. Then
$$ \sum_{f \in S_{k}} \frac{1}{\|f\|^{2}} \sum_{p \leq N} \big(\lambda_{f}(p)\big)^{4} = \sum_{f \in S_{k}} \frac{1}{\|f\|^{2}} \big(2+O(\frac{N (\log N)^{2}}{k^{\frac{1}{2}}})\big) \big(\sum_{p \leq N} 1 \big)$$
as $N \rightarrow \infty.$
\end{lemma}

\begin{proof}

By Hecke relations, $\lambda_{f}(p)^{4} = 1+2\lambda_{f}(p^{2})+\lambda_{f}(p^{2})^{2}.$ From the main terms in (2.4) and (2.5), we get $$\sum_{f \in S_{k}} \frac{2}{\|f\|^{2}}\big(\sum_{p \leq N} 1 \big),$$
and the remainder terms come from the error terms in (2.3), (2.4), (2.5).
\end{proof}

\section{Main Theorem}
For convenience, let 
\begin{equation}
A_{f}(N,h):= \sum_{N-h \leq p \leq N} \lambda_{f}(p)^{2}.\end{equation}
\begin{proposition} Let $\epsilon>0, \eta>0$ be fixed positive numbers. Let $N$ be a sufficiently big number, and $y=\eta \log N.$ Let $k$ be an even number such that $\thinspace N^{2+\epsilon}=o_{\epsilon}(k).$
Then
\begin{equation}
 \sum_{f \in S_{k}}\frac{1}{ \|f\|^{2}} \sum_{n \leq N} A_{f}(n,y)^{2} =  \sum_{f \in S_{k}}\frac{1}{ \|f\|^{2}}\sum_{n \leq N} \Big(\big(\pi(n-y)-\pi(n)\big)^{2}+\eta+o_{\epsilon}(1)\Big)
\nonumber.  
\end{equation}

\end{proposition}
\begin{proof}
By the definition (3.1), 
\begin{equation} 
\sum_{n \leq N} A_{f}(N,y)^{2} = \sum_{n\leq N} \sum_{1 \leq d_{1},d_{2} \leq y} \lambda_{f}(n-d_{1})^{2}1_{n-d_{1} \in P} \lambda_{f}(n-d_{2})^{2}1_{n-d_{2} \in P}.\end{equation}
Let's split the inner sum into diagonal terms and off-diagonal terms 
\begin{equation}\sum_{n-y \leq p \leq n} \lambda_{f}(p)^{4} + 2\sum_{1 \leq d_{1}<d_{2} \leq y} \lambda_{f}(n-d_{1})^{2}1_{n-d_{1} \in P} \lambda_{f}(n-d_{2})^{2}1_{n-d_{2} \in P}. \end{equation} 
First, let's consider the digonal terms $$\sum_{n\leq N} \sum_{n-y \leq p \leq n} \lambda_{f}(p)^{4}.$$ 
By Hecke relations, $$\lambda_{f}(p)^{4}= \lambda_{f}(p^{2})^{2}+ 2\lambda_{f}(p^{2})+1.$$
By the assumption $N^{2+\epsilon}=o_{\epsilon}(k),$ Lemma 2.3,
\begin{equation}\begin{split} \sum_{f \in S_{k}} &\frac{1}{\|f\|^{2}} \sum_{n \leq N} \sum_{n-y \leq p \leq n}\big(\lambda_{f}(p)\big)^{4} \\&= \sum_{f \in S_{k}} \frac{1}{\|f\|^{2}} \big(2+o(1)\big)  \big(\sum_{n \leq N}\sum_{n-y \leq p \leq n} 1 \big)\\&= \sum_{f \in S_{k}} \frac{1}{\|f\|^{2}} \big(1+o(1)\big) \big(\sum_{n \leq N}\sum_{n-y \leq p \leq n} 1 \big) + \sum_{f \in S_{k}} \frac{1}{\|f\|^{2}}\big(\sum_{n \leq N} \sum_{n-y \leq p \leq n} 1 \big)\\&=  \sum_{f \in S_{k}} \frac{1}{\|f\|^{2}} \big(1+o(1)\big) \big(\sum_{n \leq N}\sum_{n-y \leq p \leq n} 1 \big)+\sum_{f \in S_{k}} \frac{1}{\|f\|^{2}}  \sigma(2,1)\sum_{\vec{d}}\pi_{\vec{d}}(N)\end{split} \end{equation} where the summation over $\vec{d}$ means the sum over $\vec{d}=(d) \in \mathbb{N}, 1\leq d \leq y$. 
Let's consider the off-diagonal terms
$$ \sum_{n \leq N}\sum_{1 \leq d_{1}<d_{2} \leq y} \lambda_{f}(n-d_{1})^{2} \lambda_{f}(n-d_{2})^{2}1_{n-d_{1},n-d_{2} \in P}.$$
By (2.1),
\begin{equation}\begin{split}  \lambda_{f}(n-d_{1})^{2} &\lambda_{f}(n-d_{2})^{2}1_{n-d_{1},n-d_{2} \in P}\\&=\lambda_{f}((n-d_{1})^{2})1_{n-d_{1} \in P}+\lambda_{f}((n-d_{2})^{2})1_{n-d_{2} \in P}\\&\;\;\;+\lambda_{f}((n-d_{1})^{2}) \lambda_{f}((n-d_{2})^{2})1_{n-d_{1},n-d_{2} \in P}\\&\;\;\;+1_{n-d_{1},n-d_{2} \in P}.\end{split}\nonumber \end{equation}
From $1_{n-d_{1},n-d_{2} \in P}$ in the above equation, we get   \begin{equation} 2\sum_{f \in S_{k}} \frac{1}{\|f\|^{2}} \sum_{n \leq N} \sum_{1\leq d_{1}<d_{2} \leq y} 1_{n-d_{1},n-d_{2} \in P} = \sum_{f \in S_{k}} \frac{1}{\|f\|^{2}} \sigma(2,2) \sum_{\vec{d}} \pi_{\vec{d}}(N) \end{equation} where the summation over $\vec{d}$ means the summation over $\vec{d}=(d_{1},d_{2}) \in \mathbb{N}^{2},$ $1\leq d_{1}<d_{2}\leq y.$
By (2.3),
\begin{equation}\begin{split}\sum_{f \in S_{k}}&\frac{1}{\|f\|^{2}} \sum_{n \leq N} \sum_{1 \leq d_{1} <d_{2} \leq y} \lambda_{f}((n-d_{1})^{2})1_{n-d_{1} \in P}+\lambda_{f}((n-d_{2})^{2})1_{n-d_{2} \in P}
\\&=  \sum_{n \leq N} \sum_{1 \leq d_{1} <d_{2} \leq y}   \sum_{f \in S_{k}}\frac{1}{\|f\|^{2}}  \lambda_{f}((n-d_{1})^{2})1_{n-d_{1} \in P}+\lambda_{f}((n-d_{2})^{2})1_{n-d_{2} \in P}
\\& =O_{\epsilon}\Big(\frac{(4\pi)^{k-1}}{\Gamma(k-1) k^{\frac{1}{2}}}N^{\frac{3}{2}+\epsilon}y^{2}\Big).\end{split} \end{equation}
By the assumption $N^{2+\epsilon}=o_{\epsilon}(k)$ and (2.5),
$$\frac{(4\pi)^{k-1}}{\Gamma(k-1) k^{\frac{1}{2}}}N^{\frac{1}{2}+\epsilon} = o \Big( \sum_{f \in S_{k}}\frac{1}{\|f\|^{2}}\Big).$$
Therefore, (3.6) is bounded by $$o_{\epsilon}\Big(N\sum_{f \in S_{k}} \frac{1}{\|f\|^{2}}\Big)$$ for sufficiently big $N.$
By Lemma 2.1,
\begin{equation}\begin{split} \sum_{f \in S_{k}}&\frac{1}{\|f\|^{2}}\lambda_{f}((n-d_{1})^{2}) \lambda_{f}((n-d_{2})^{2})1_{n-d_{1},n-d_{2} \in P}\\&=O\big((\log N)^{2}\frac{N(4\pi)^{k-1}}{\Gamma(k-1)k^{\frac{1}{2}}}1_{n-d_{1},n-d_{2} \in P}\big).\end{split}\end{equation}
Therefore by the assumption $N^{2+\epsilon}=o_{\epsilon}(k)$ and (2.5),
\begin{equation}\begin{split} \sum_{1 \leq d_{1} < d_{2} \leq y} \sum_{n \leq N}&\sum_{f \in S_{k}}\frac{1}{\|f\|^{2}}\lambda_{f}\big((n-d_{1})^{2}\big) \lambda_{f}\big((n-d_{2})^{2}\big)1_{n-d_{1},n-d_{2} \in P} \\&=O\Big((\log N)^{2}\frac{N (4\pi)^{k-1}}{\Gamma(k-1)k^{\frac{1}{2}}}\sum_{1 \leq d_{1} <d_{2}  \leq y}\sum_{n \leq N}1_{n-d_{1},n-d_{2} \in P}\Big) \\&=o_{\epsilon} \Big(\sum_{f \in S_{k}} \frac{1}{\| f\|^{2}}\sum_{1 \leq d_{1} <d_{2}  \leq y}\sum_{n \leq N}1_{n-d_{1},n-d_{2} \in P}\Big).\end{split} \end{equation}
Hence,  
\begin{equation}\begin{split}
\sum_{f \in S_{k}}\frac{1}{\|f\|^{2}}\sum_{n \leq N} A_{f}(n,y)^{2}&=
\Big(\sum_{f \in S_{k}}\frac{1}{\|f\|^{2}}\sum_{n \leq N} (\pi(n-y)-\pi(n))^{2}\Big)\big(1+o_{\epsilon}(1)\big) \\&+ \sum_{f \in S_{k}} \frac{1}{\|f\|^{2}}\frac{yN}{\log N}\big(1+o_{\epsilon}(1)\big).
\end{split}\end{equation}
\end{proof}
By applying the Hardy-Littlewood prime $2$-tuples conjecture (the last equation in (1.5)), we get the following theorem.
\begin{theorem}Assume the Hardy-Littlewood prime $2$-tuples conjecture. Let $\epsilon>0, \eta>0$ be fixed positive numbers. Let $N$ be a sufficiently big number, and $y=\eta \log N.$ Let $k$ be an even number such that $\thinspace N^{2+\epsilon}=o_{\epsilon}(k).$
Then
\begin{equation} \begin{split}
\frac{1}{N} \sum_{f \in S_{k}}\frac{1}{ \|f\|^{2}} \sum_{n \leq N}  \Big(\sum_{n-y \leq p \leq n} \lambda_{f}(p)^{2}\Big)^{2} &=  \sum_{f \in S_{k}}\frac{1}{ \|f\|^{2}}\big(m_{2}(\eta)+\eta+o_{\epsilon}(1)\big) \\&=\frac{(4\pi)^{k-1}}{\Gamma(k-1)}\big(m_{2}(\eta)+\eta+o_{\epsilon}(1)\big).
\nonumber  
\end{split}\end{equation}
\end{theorem}
\begin{remark}
It is known that $|S_{k}| \sim \frac{k}{12}.$ 
Let $\Lambda_{f}(p)= \frac{\Gamma(k-1)^{\frac{1}{4}}k^{\frac{1}{4}}\lambda_{f}(p)}  {12^{\frac{1}{4}}\|f\|^{\frac{1}{2}}(4\pi)^{\frac{k-1}{4}}}.$ Then with the conditions in Theorem 3.2, 

\begin{equation} \frac{1}{N} \sum_{f \in S_{k}} \sum_{n \leq N}  \Big(\sum_{n-y \leq p \leq n} \Lambda_{f}(p)^{2}\Big)^{2} =\sum_{f \in S_{k}} \Big(m_{2}(\eta)+\eta+o_{\epsilon}(1)\Big). \end{equation}
\end{remark}
\section{Proof of Theorem 1.3, Theorem 1.5}
\subsection{Proof of Theorem 1.3}
\begin{proof}
The proof is basically same as the proof of Theorem 3.2.
For the diagonal terms, we just need to replace (3.4) with 
\begin{equation}\begin{split}
\sum_{j} \sum_{n \leq N} & \sum_{n-y \leq p \leq n}\ \lambda_{\psi_{j}}(p)^{4} \frac{\zeta(2)}{L(1,\mathrm{sym}^{2}\psi_{j})} e^{-\frac{t_{j}}{T}} \\&= \sum_{j}  \frac{\zeta(2)}{L(1,\mathrm{sym}^{2}\psi_{j})} e^{-\frac{t_{j}}{T}}\big(2+o(1)\big)\big(\sum_{n \leq N}  \sum_{n-y \leq p \leq n}\ 1\big)\end{split}\end{equation} (By Lemma 1.2 and the assumption  $\thinspace N^{1+\epsilon}=o_{\epsilon}(T)$).
For the off-diagonal terms, we need to replace (3.6) with 
\begin{equation} \begin{split}
\sum_{j}  \frac{\zeta(2)e^{-\frac{t_{j}}{T}}}{L(1,\mathrm{sym}^{2}\psi_{j})} &\sum_{n \leq N} \sum_{1 \leq d_{1} <d_{2} \leq y} \lambda_{\psi_{j}}((n-d_{1})^{2})1_{n-d_{1} \in P}+\lambda_{\psi_{j}}((n-d_{2})^{2})1_{n-d_{2} \in P}\\&=O_{\epsilon}\Big((T^{1+\epsilon}N^{\epsilon} +N^{1+\epsilon})y^{2}N\Big).
\end{split} \end{equation}
And then, we need to replace (3.7) with
\begin{equation}\sum_{j}\frac{\zeta(2)e^{-\frac{t_{j}}{T}}}{L(1,\mathrm{sym}^{2}\psi_{j})} \lambda_{\psi_{j}}((n-d_{1})^{2}) \lambda_{\psi_{j}}((n-d_{2})^{2})1_{n-d_{1},n-d_{2} \in P}=O_{\epsilon}\big(T^{1+\epsilon}N^{\epsilon} +N^{2+\epsilon}).\end{equation}
Therefore by the assumption $\thinspace N^{1+\epsilon}=o_{\epsilon}(T)$ and (4.3),
\begin{equation}\begin{split} & \sum_{1 \leq d_{1} < d_{2} \leq y} \sum_{n \leq N} \sum_{j}  \frac{\zeta(2)e^{-\frac{t_{j}}{T}}}{L(1,\mathrm{sym}^{2}\psi_{j})}  \lambda_{\psi_{j}}\big((n-d_{1})^{2}\big) \lambda_{\psi_{j}}\big((n-d_{2})^{2}\big)1_{n-d_{1},n-d_{2} \in P} \\&=O_{\epsilon}\Big((T^{1+\epsilon}N^{\epsilon}+N^{2+\epsilon})\sum_{1 \leq d_{1} <d_{2}  \leq y}\sum_{n \leq N}1_{n-d_{1},n-d_{2} \in P}\Big) \\&=o_{\epsilon} \Big(\sum_{j}\frac{\zeta(2)e^{-\frac{t_{j}}{T}}}{L(1,\mathrm{sym}^{2}\psi_{j})} \sum_{1 \leq d_{1} <d_{2}  \leq y}\sum_{n \leq N}1_{n-d_{1},n-d_{2} \in P}\Big).\end{split}\end{equation}
Finally, 
\begin{equation} \begin{split}
 \sum_{j} \frac{\zeta(2)e^{-\frac{t_{j}}{T}}}{L(1,\mathrm{sym}^{2}\psi_{j})} &  \sum_{n \leq N} \Big(\sum_{n-y \leq p \leq n} \lambda_{\psi_{j}}(p)^{2} \Big)^{2}
\\ & =\sum_{j} \frac{\zeta(2)e^{-\frac{t_{j}}{T}}}{L(1,\mathrm{sym}^{2}\psi_{j})}  \sum_{n \leq N} \big((\pi(n-y)-\pi(n))^{2}+\eta+o_{\epsilon}(1)\big),\end{split}\end{equation}
and by the Hardy-Littlewood prime 2-tuples conjecture,
 $$\sum_{n \leq N} \big((\pi(n-y)-\pi(n))^{2}+\eta+o_{\epsilon}(1)\big)=N\big(m_{2}(\eta)+\eta+o_{\epsilon}(1)\big).$$
\end{proof}
\subsection{Proof of Theorem 1.5}
\begin{proof}
  By squaring out, 
\begin{equation} \begin{split}
\sum_{n \leq N}&\Big(\sum_{n-y \leq p \leq n} |A_{j}(p,1)|^{2}\Big)^{2} \\&= \sum_{n\leq N} \sum_{1 \leq d_{1},d_{2} \leq y} |A_{j}(n-d_{1},1)|^{2}1_{n-d_{1} \in P} |A_{{j}}(n-d_{2},1)|^{2}1_{n-d_{2} \in P}.\end{split}\end{equation}
We split the inner sum into digonal terms and off-diagonal terms 
\begin{equation}\sum_{n-y \leq p \leq n} |A_{j}(p,1)|^{4}+ 2\sum_{1 \leq d_{1}<d_{2} \leq y} |A_{j}(n-d_{1},1)|^{2}1_{n-d_{1} \in P} |A_{j}(n-d_{2},1)|^{2}1_{n-d_{2} \in P}. \end{equation} 
First, let's consider the digonal terms $$\sum_{n\leq N}\sum_{n-y \leq p \leq n} |A_{j}(p,1)|^{4}.$$ 
By (1.9), Lemma 1.4 and the assumption $N^{\frac{32}{3}+2\epsilon}=o_{\epsilon}(T),$
\begin{equation}\begin{split}\sum_{j}&\frac{1}{\underset{s=1}{\mathrm{Res}}\; L(s,\phi_{j} \times \bar{\phi}_{j})} e^{-\frac{v_{j}}{T^{2}}} \sum_{n \leq N}  \sum_{n-y \leq p \leq n}\ |A_{j}(p,1)|^{4} \\&= \sum_{j}\frac{1}{\underset{s=1}{\mathrm{Res}}\; L(s,\phi_{j} \times \bar{\phi}_{j})} e^{-\frac{v_{j}}{T^{2}}}  \big(2+o(1)\big)  \big(\sum_{n \leq N}  \sum_{n-y \leq p \leq n}\ 1 \big).\end{split} \end{equation}
Let's consider the off-diagonal terms
$$ \sum_{n \leq N}\sum_{1 \leq d_{1}<d_{2} \leq y} |A_{j}(n-d_{1},1)|^{2}1_{n-d_{1} \in P} |A_{j}(n-d_{2},1)|^{2}1_{n-d_{2} \in P}.$$
By (1.9),
\begin{equation}\begin{split}  |A_{j}(n-d_{1},1)|^{2} & |A_{j}(n-d_{2},1)|^{2}1_{n-d_{1},n-d_{2} \in P}=A_{j}(n-d_{1},n-d_{1})1_{n-d_{1} \in P}\\&+A_{j}(n-d_{2},n-d_{2})1_{n-d_{2} \in P}\\&+A_{j}(n-d_{1},n-d_{1}) A_{j}(n-d_{2},n-d_{2})1_{n-d_{1},n-d_{2} \in P}\\&+1_{n-d_{1},n-d_{2} \in P}.\end{split}\nonumber \end{equation}
From $1_{n-d_{1},n-d_{2} \in P}$ in the above equation, we get  \begin{equation} \begin{split} 2\sum_{j}&\frac{e^{-\frac{v_{j}}{T^{2}}}}{\underset{s=1}{\mathrm{Res}}\; L(s,\phi_{j} \times \bar{\phi}_{j})} \sum_{n \leq N} \sum_{1\leq d_{1}<d_{2} \leq y} 1_{n-d_{1},n-d_{2} \in P} \\&= \sum_{j}\frac{e^{-\frac{v_{j}}{T^{2}}}}{\underset{s=1}{\mathrm{Res}}\; L(s,\phi_{j} \times \bar{\phi}_{j})}\sigma(2,2) \sum_{\vec{d}} \pi_{\vec{d}}(N).\end{split}\end{equation} where the summation over $\vec{d}$ means the summation over $\vec{d}=(d_{1},d_{2}) \in \mathbb{N}^{2},$ $1\leq d_{1}<d_{2}\leq y.$
By Lemma 1.4, 
\begin{equation}\begin{split}\sum_{j}&\frac{e^{-\frac{v_{j}}{T^{2}}}}{\underset{s=1}{\mathrm{Res}}\; L(s, \phi_{j} \times \bar{\phi}_{j})}\sum_{n \leq N} \sum_{1 \leq d_{1} <d_{2} \leq y}\Big( A_{j}(n-d_{1},n-d_{1})1_{n-d_{1} \in P} \\&+ A_{j}(n-d_{2},n-d_{2})1_{n-d_{2} \in P} \Big) \\&=O_{\epsilon}\Big(y^{2}N^{3+\epsilon}T^{\frac{37}{8}+\epsilon}\Big). \end{split}\end{equation}
By the assumption $N^{\frac{32}{3}+2\epsilon}=o_{\epsilon}(T)$ and (1.10), (4.10) is bounded by 
$$o_{\epsilon}\Big(N\sum_{j} \frac{e^{-\frac{v_{j}}{T^{2}}}}{\underset{s=1}{\mathrm{Res}}\; L(s, \phi_{j} \times \bar{\phi}_{j})}\Big)$$ for sufficiently big $N.$
By Lemma 1.4,
\begin{equation}\begin{split} \sum_{j}&\frac{e^{-\frac{v_{j}}{T^{2}}}}{\underset{s=1}{\mathrm{Res}}\; L(s, \phi_{j} \times \bar{\phi}_{j})} A_{j}(n-d_{1},n-d_{1}) A_{j}(n-d_{2},n-d_{2})1_{n-d_{1},n-d_{2} \in P}\\&=O_{\epsilon}\big(N^{4+\epsilon} T^{\frac{37}{8}+\epsilon} 1_{n-d_{1},n-d_{2} \in P}\big). \end{split} \end{equation}
Therefore by the assumption $N^{\frac{32}{3}+2\epsilon}=o_{\epsilon}(T),$
\begin{equation}\begin{split} \sum_{1 \leq d_{1} < d_{2} \leq y} &\sum_{n \leq N}\sum_{j}\frac{e^{-\frac{v_{j}}{T^{2}}} A_{j}(n-d_{1},n-d_{1})A_{j}(n-d_{2},n-d_{2})}{\underset{s=1}{\mathrm{Res}}\; L(s, \phi_{j} \times \bar{\phi}_{j})}1_{n-d_{1},n-d_{2} \in P} \\&=O\Big(N^{4+\epsilon}T^{\frac{37}{8}+\epsilon} \sum_{1 \leq d_{1} <d_{2}  \leq y}\sum_{n \leq N}1_{n-d_{1},n-d_{2} \in P}\Big) \\&=o_{\epsilon} \Big(\sum_{j}\frac{e^{-\frac{v_{j}}{T^{2}}}}{\underset{s=1}{\mathrm{Res}}\; L(s, \phi_{j} \times \bar{\phi}_{j})}\sum_{1 \leq d_{1} <d_{2}  \leq y}\sum_{n \leq N}1_{n-d_{1},n-d_{2} \in P}\Big).\end{split} \end{equation}
Finally,
\begin{equation}\begin{split}
\sum_{j}&\frac{e^{-\frac{v_{j}}{T^{2}}}}{\underset{s=1}{\mathrm{Res}}\; L(s, \phi_{j} \times \bar{\phi}_{j})} \sum_{n \leq N} \Big(\sum_{n-y \leq p \leq n} |A_{j}(p,1)|^{2} \Big)^{2}\\&=
\Big(\sum_{j}\frac{e^{-\frac{v_{j}}{T^{2}}}}{\underset{s=1}{\mathrm{Res}}\; L(s, \phi_{j} \times \bar{\phi}_{j})}\sum_{n \leq N} \big(\pi(n-y)-\pi(n)\big)^{2}\Big)\big(1+o_{\epsilon}(1)\big) \\&+ \sum_{j}\frac{e^{-\frac{v_{j}}{T^{2}}}}{\underset{s=1}{\mathrm{Res}}\; L(s, \phi_{j} \times \bar{\phi}_{j})}\frac{yN}{\log N}\big(1+o_{\epsilon}(1)\big),
\end{split}\end{equation}
and by the Hardy-Littlewood prime 2-tuples conjecture,
 $$\sum_{n \leq N} \big((\pi(n-y)-\pi(n)\big)^{2}+\eta+o_{\epsilon}(1)\big)=N\big(m_{2}(\eta)+\eta+o_{\epsilon}(1)\big).$$
\end{proof}

\subsection*{Acknowledgements} The author would like to thank his advisor Professor Xiaoqing Li, for her constant support. The author would also like to thank Professor Milicevic for pointing out that the Hardy-Littlewood conjecture was not included in the statement of Theorem 3.1 in the previous version, and the referee for useful suggestions.  

\bibliographystyle{siam}   
\bibliography{ss1}  
\end{document}